\documentclass[11pt,reqno]{amsart}
\usepackage{amsmath, amsfonts, color, amsthm, graphics, amssymb,url,mathrsfs}
\usepackage{stmaryrd}
\usepackage{cleveref}

\usepackage[all]{xy}

\theoremstyle{plain}
\newtheorem{thm}{Theorem}[section]
\newtheorem{lem}[thm]{Lemma}
\newtheorem{prop}[thm]{Proposition}
\newtheorem{cor}[thm]{Corollary}

\theoremstyle{definition}
\newtheorem*{Ack}{Acknowledgement}
\newtheorem{deff}[thm]{Definition}

\newtheorem{remark}[thm]{Remark}
\newtheorem{question}[thm]{Question}

\theoremstyle{remark}

\def\k{\ensuremath{\bold{k}}}

\def\U{\mathcal{U}}

\newcommand*{\Hom}{\ensuremath{\text{\upshape Hom}}}
\newcommand*{\Homology}{\ensuremath{H}}
\newcommand*{\RHom}{\ensuremath{\text{\upshape RHom}}}
\newcommand*{\HP}{\ensuremath{HP}}
\newcommand*{\HH}{\ensuremath{HH}}

\newcommand*{\gr}{\ensuremath{\text{\upshape gr}}}
\newcommand*{\Ext}{\ensuremath{\text{\upshape Ext}}}
\newcommand*{\Tor}{\ensuremath{\text{\upshape Tor}}}

\newcommand*{\Spec}{\ensuremath{\text{\upshape Spec}}}
\newcommand*{\Aut}{\ensuremath{\text{\upshape Aut}}}

\newcommand*{\Der}{\ensuremath{\text{\upshape Der}}}
\newcommand*{\Pic}{\ensuremath{\text{\upshape Pic}}}

\newcommand*{\Id}{\ensuremath{\text{\upshape Id}}}

\def\dim{\operatorname{dim}}

\begin{document}
\thispagestyle{empty}

\title{Homological unimodularity and Calabi-Yau condition for Poisson algebras}

\author{Jiafeng L\"u}
\address{L\"u: Department of Mathematics, Zhejiang Normal University, Jinhua, Zhejiang 321004, P.R. China}
\email{jiafenglv@zjnu.edu.cn, jiafenglv@gmail.com}

\author{Xingting Wang}
\address{Wang: Department of Mathematics, Temple University, Philadelphia 19122, USA }
\email{xingting@temple.edu}

\author{Guangbin Zhuang}
\address{Zhuang: Department of Mathematics,
University of Southern California, Los Angeles 90089-2532, USA}
\email{gzhuang@usc.edu}

\begin{abstract}
In this paper, we show that the twisted Poincar\'e duality between Poisson homology and cohomology can be derived from the Serre invertible bimodule. This gives another definition of a unimodular Poisson algebra in terms of its Poisson Picard group. We also achieve twisted Poincar\'e duality for Hochschild (co)homology of Poisson bimodules using rigid dualizing complex. For a smooth Poisson affine variety with the trivial canonical bundle, we prove that its enveloping algebra is a Calabi-Yau algebra if the Poisson structure is unimodular.
\end{abstract}
\maketitle

\section{Introduction}
Poisson geometry is originated in classical mechanics where one describes the time evolution of a mechanical system by solving HamiltonÕs equations in terms of the Hamiltonian vector field.  This inspires the definition of a Poisson manifold $M$ which is equipped with a Lie bracket (called Poisson bracket) on the algebra $C^\infty(M)$ of smooth functions on $M$ subject to the Leibniz rule. From an algebraic point of view, the structure of a Poisson bracket is understood through the following definition of a Poisson algebra.
\begin{deff}
A Poisson algebra is a commutative algebra $A$ over a base field $k$, which is equipped with a bilinear map $\{-,-\}: A\otimes_k A\to A$ satisfying
\begin{enumerate}
\item skew symmetry: $\{a,b\}=-\{b,a\}$,
\item Jacobi identity: $\{a,\{b,c\}\}+\{b,\{c,a\}\}+\{c,\{a,b\}\}=0$,
\item Leibniz rule: $\{ab,c\}=a\{b,c\}+\{a,c\}b$,
\end{enumerate}
for all $a,b,c\in A$.
\end{deff}
Recently, the development of Poisson geometry has deeply entangled with noncommutative algebras and noncommutative geometry. For instance, in the deformation quantization of a Poisson algebra, the Poisson (co)homology of the Poisson algebra and the Hochschild (co)homology of its deformation quantization are connected by the Brylinsky spectral sequence \cite{Bry88, EG}. As an application, Van den Bergh \cite{VdB94}, Marconnet \cite{Mar04} and Berger-Pichereau \cite{BP11} computed the Hochschild homology of some three-dimensional Calabi-Yau algebras, by considering them as deformations of polynomial Poisson algebras with three variables  respectively and computing the corresponding Poisson homology.

This prompts us to study the representation theory of a Poisson algebra $A$. Let us first recall what happens in Poisson manifold. Let $E\to M$ be a vector bundle on a Poisson Manifold $M$, and let $\Gamma^\infty (E)$ be the space of smooth sections of $E$, regarded as a right finitely generated projective module over $C^\infty(M)$. Then any Poisson vector bundle structure on $E\to M$ is a Poisson $\Gamma^\infty(M)$-module structure on $\Gamma^\infty(E)$, or equivalently, it is a flat contravariant connection on $E$. Now we provide its algebraic version.

\begin{deff}\label{PMod}\cite{Oh} Let $A$ be a Poisson algebra over a base field $k$. A left {\it Poisson module} $M$ over $A$ is a left $A$-module with a linear map
$$\{-, -\}_M: A\otimes_k M\rightarrow M$$
satisfying
\begin{enumerate}
\item $\{ \{a, b\}_A, m\}_M=\{a, \{b, m\}_M\}_M-\{b, \{a, m\}_M\}_M,$
\item $\{ab, m\}_M=a\{b, m\}_M+b\{a, m\}_M$,
\item $\{a, bm\}_M=\{a, b\}m+b\{a, m\}_M$,
\end{enumerate}
for any $a, b\in A$ and $m\in M$. For two left Poisson $A$-modules $M$ and $N$, a {\it Poisson module morphism} $f: M\rightarrow N$ is an $A$-module map $f$ such that
$$f(\{a, m\}_M)=\{a, f(m)\}_N$$ for any $a\in A$ and $m\in M$. We define $A^{op}$ to be the {\it opposite Poisson algebra} of $A$, where $A^{op}=A$ as commutative algebras and $\{-,-\}_{A^{op}}=-\{-,-\}_A$. Similarly, one can define right Poisson $A$-modules to be left Poisson $A^{op}$-modules. We denote by $A\text{-PMod}$ (resp. $A^{op}\text{-PMod}$) the category of left (resp. right) Poisson modules over $A$.
\end{deff}

The first important result of the representation theory of a Poisson algebra $A$ is the following equivalence of categories, which enables us to express the Poisson homology and cohomology in terms of the torsion and extension groups via the universal enveloping algebra of $A$; see Definition \ref{UnivP}.
\begin{thm}\cite[Corollary 1]{Um}\label{T:Equi}
There is an equivalence of categories $A\text{-PMod}\equiv \U(A)\text{-Mod}$, where $\U(A)$ is the enveloping algebra of the Poisson algebra $A$.
\end{thm}


Regarding the homological behavior of a Poisson algebra, the phenomena of twisted Poincar\'e duality between Poisson homology and cohomology has been widely observed in many Poisson structures, i.e., polynomial Poisson algebras with quadratic Poisson structures \cite{LR07}, affine Poisson algebras \cite{Zhu} and later for any polynomial Poisson algebras \cite{LWW}. This twisted duality enables us to decode the rich information of Poisson structures, that is carried by Poisson cohomology but hard to compute, through Poisson homology that is sometimes more computable.

In the light of Proposition \ref{P:THPoisson}, we see that the representation category of a Poisson algebra $A$ is a monoidal category, where the tensor product is given by $\otimes_A$ and the identity object is the trivial Poisson module $A$. Moreover, for any Poisson module $M$ that is a line bundle over $A$ (locally free of rank one), we have its left and right dual given by $M^*:=\Hom_A(A,M)$; see \Cref{RankonePMod}. This yields the definition of the Poisson Picard group of $A$ in the following way.

\begin{deff}\label{D:PicardG}
Let $A$ be a Poisson algebra. We define the {\it Poisson Picard group} of $A$, $\Pic_P(A)$ to be the set of isomorphism classes of Poisson modules that are line bundles over $A$, with multiplication given by $\otimes_A$ and inverses given by $M\to M^*$.
\end{deff}

Throughout the paper, we are interested in affine smooth Poisson algebra $A$. In this case, the differential forms of maximal degree $\omega_A=\Omega_A^\ell$ for $\ell=\dim A$ is a line bundle over $A$. Moreover, its inverse is given by $\omega_A^*=\Ext_{A\otimes A}^\ell(A,A\otimes A)$, which turns out to be the Serre invertible bimodule; see subsection \ref{S:Rigid}.  According to Lemma \ref{PAExt}, both $\omega_A$ and $\omega_A^*$ are equipped with Poisson module structures; and hence belong to the Poisson Picard group $\Pic_P(A)$. The following theorem shows that $\omega_A^*$ plays an important role in the twisted Poincar\'e duality for Poisson homology and cohomology.

\begin{thm}[Theorem \ref{TwistedP}]\label{Intro:T1}
Let $A$ be an affine smooth Poisson algebra of dimension $\ell$. Then for any right Poisson $A$-module $M$, we have
\[
\HP^{i}(M)=\HP_{\ell-i}(M\otimes_A \omega_A^*),
\]
where $\HP^i(M)$ is the $i$-th Poisson cohomology of $A$ with values in $M$, and $\HP_{\ell-i}(M)$ is the $(\ell-i)$-th Poisson homology of $A$ with values in the tensor product of Poisson modules $M\otimes_A\omega^*$.
\end{thm}
In particular when $A$ has trivial canonical bundle, or $A$ is a commutative Calabi-Yau algebra by Proposition \ref{ComCY}, we can recover the modular derivation of $A$ \cite[\S 2.2]{LWW} from the Poisson module structure on $\omega_A$ by Lemma \ref{L:APDer}. In this case, $M\otimes_A \omega_A^*=M^{\delta}$, where $M^{\delta}$ denotes the twisted Poisson module of $M$ by the modular derivation $\delta$ of $A$; refer to Corollary \ref{C:TwistD}.

Our next goal is to explore the unimodularity of a Poisson algebra. In the seminal paper \cite{Wein}, Weinstein introduced a notion of modular class of a smooth real Poisson manifold $M$, which belongs to the $1$-th Poisson cohomology group $\HP^1(M)$. The notion was independently introduced by Brylinski-Zuckerman \cite{BZ99} in the context of complex analytic Poisson manifold. We say a Poisson manifold is unimodular if the modular class of $M$ equals zero in $\HP^1(M)$. Later in \cite{Xu99}, Xu proved that any Poisson manifold that is unimodular satisfies  Poincar\'e duality between Poisson homology and cohomology.

Let us return to the algebraic setting. Let $A$ be a smooth Poisson algebra that has trivial canonical bundle. When we regard the unimodularity of $A$, the $1$-th Poisson cohomology group
\[\HP^1(A)=\{\text{Poisson derivations} \}/\{\text{Hamiltonian derivations}\}\]
needs to be replaced by
\[\mathscr{HP}^1(A)=\{\text{Poisson derivations} \}/\{\text{log-Hamiltonian derivations}\}.\]
Then we can say that $A$ is unimodular if the modular class of $A$, which is represented by the modular derivation of $A$, equals zero in $\mathscr{HP}^1(A)$. See \cite{Dol,LWW}. Our Theorem \ref{Intro:T1} reveals a closed connection between unimodularity and (untwisted) Poincar\'e duality between Poisson homology and cohomology. Hence we provide a version of unimodularity in terms of the Poisson Picard group.
\begin{deff}[Definition \ref{Hunimod}]
Let $A$ be an affine smooth Poisson algebra. We say $A$ is {\it homologically unimodular} if the class of $\omega_A$ equals zero in the Poisson Picard group $\Pic_P(A)$.
\end{deff}
Finally, we study the connection between unimodularity of a Poisson algebra and Calabi-Yau condition of its enveloping algebra. Grothendieck in \cite{Hart} introduced dualizing complex to develop duality theory for singular curves. The noncommutative version of the dualizing complex was introduced by Yekutieli in \cite{Yek}, and it becomes one of the standard homological tools of noncommutative ring theory. Van den Bergh later defined rigid dualizing complex for any noetherian ring \cite{VDB}. It appears that the existence of the rigid dualizing complex has closed relationship with Calabi-Yau algebra defined by Ginzburg \cite{Gin}. See Proposition \ref{CYRigid}. The following result is obtained by applying the formula of the rigid dualizing complex for algebras of generalized differential operators \cite{Chem99}.

\begin{thm}[Theorem \ref{RigidDCP}]\label{Intro:T2}
Let $A$ be a smooth affine Poisson algebra of dimension $\ell$. The rigid dualizing complex of the enveloping algebra $\U(A)$ is
\[
\U(A)\otimes_A\mathcal L_A\,[2\ell],
\]
where $\mathcal L_A:=\omega_A\otimes_A\omega_A$.
\end{thm}
Therefore the special class $\omega_A$ in the Poisson Picard group $\Pic_P(A)$ contributes in the rigid dualizing complex of $\U(A)$ by a factor of two. As a consequence, the class $\mathcal L_A\in \Pic_P(A)$ is involved in the twisted Poincar\'e duality between Hochschild homology and cohomology of $\U(A)$. Indeed, $\mathcal L_A^*\otimes_A\U(A)$ turns out to be the Van den Bergh invertible bimodule \cite{VDB2} due to Corollary \ref{VDBDualityH}. Moreover,  $\U(A)$ is a Calabi-Yau algebra if the Poisson structure of $A$ is unimodular.

\begin{cor}[\Cref{SCYP}\&\Cref{CYUni}]
Let $A$ be a Calabi-Yau Poisson algebra. Then the enveloping algebra $\U(A)$ is skew Calabi-Yau of dimension $2\ell$ with Nakayama automorphism given by $2\delta$, where $\delta$ is the modular derivation of $A$. Moreover, $\U(A)$ is Calabi-Yau if the Poisson structure of $A$ is unimodular.
\end{cor}
In the deformation quantization of a Poisson algebra, Dolgushev showed that the deformation quantization algebra is a Calabi-Yau algebra if and only if the corresponding Poisson structure is unimodular \cite{Dol}. Our result further confirms that unimodularity of a Poisson algebra plays the same role as Calabi-Yau condition in its representation category.

We want to make a remark that a general duality theorem is proved by Huebschmann \cite{Hue99} and Chemia \cite{Chem94,Chem99,Chem04} in the setting of Lie-Rinehart algebras and Lie algebroids.

The paper is organized as follows. Basic definitions of Poisson (co)homology, Calabi-Yau algebra and rigid dualizing complex are recalled in Section \ref{S:Pre}. Twisted Poisson module structures are discussed in Section \ref{S:Twisted}, where we prove that the tensor $\otimes_A$ and Hom $\Hom_A(-,-)$ are two well-defined operators on the category of Poisson modules. In Section \ref{S:HomM}, we introduce the concept of homological unimodularity for any smooth affine Poisson algebra. Using the invertible Serre bimodule of a Poisson algebra, we prove the twisted Poincar\'e duality between Poisson homology and cohomology in Section \ref{S:PoinP}. At the last Section \ref{S:CY}, we establish several equivalent conditions involving unimodularity of a Poisson algebra and the Calabi-Yau condition of its enveloping algebra.

\section{Preliminary}\label{S:Pre}
Throughout the paper, we work over a base field $k$, algebraically closed of characteristic zero. The unadorned tensor product $\otimes$ means $\otimes_k$. We keep the notation $A$ as an affine Poisson algebra.

\subsection{Poisson universal enveloping algebra}
In \cite{Oh}, Oh introduced the universal enveloping algebra of $A$, denoted by $\U(A)$. Its constructive definition in terms of generators and relations is given as follows.

\begin{deff}\cite[\S 2]{Um}\label{UnivP}
Let $V=A\oplus A$ be the direct sum of two copies of $A$ with corresponding inclusions of $A$ denoted by $M$ and $H$. The {\it universal enveloping algebra} $\U(A)$ of $A$ is defined to be the tensor algebra $T\langle V\rangle $ modulo the following relations, for all $a,b\in A$,
\begin{align*}
M_{ab}&\,=M_aM_b\\
H_{\{a,b\}}&\, =H_aH_b-H_bH_a\\
H_{ab}&\, =M_aH_b+M_bH_a\\
M_{\{a,b\}}&\, =H_aM_b-M_bH_a=[H_a, M_b]\\
M_1&\, =1.
\end{align*}
\end{deff}
The presentation of $\U(A)$ results in an injective algebra map $M: A\rightarrow \U(A)$ and a Lie algebra map $H: A\rightarrow \U(A)$. Therefore we will simply consider $A$ as a subalgebra of $\U(A)$. By Theorem \ref{T:Equi}, we have an equivalence of categories $A\text{-PMod}\equiv \U(A)\text{-Mod}$. Explicitly, for any $M\in A\text{-PMod}$, we can consider $M$ as a left $\U(A)$-module where $a\cdot m=am$ and $H_a\cdot m=\{a,m\}_M$, for all $a\in A$ and $m\in M$.

Now let $B$ be another Poisson algebra. The tensor algebra $A\otimes B$ is equipped with a Poisson bracket given by
\[
\{a_1\otimes b_1,a_2\otimes b_2\}:=\{a_1,a_2\}\otimes b_1b_2+a_1a_2\otimes \{b_1,b_2\}
\]
for all $a_i\in A$ and $b_i\in B$. In particular, $A^e:=A\otimes A^{op}$ is a Poisson algebra. The following lemma is an application of the general result for DG Poisson algebras, which can be proved by the universal property of $\U(A)$; see \cite[\S 1.2]{LWZ2}.

\begin{lem}\cite[Theorem 4.5]{LWZ3}\label{TensorU}
There are algebra isomorphisms: $\U(A^{op})=\U(A)^{op}$, and $\U(A^e)=\U(A)\otimes \U(A)^{op}=:\U(A)^e$.
\end{lem}
A {\it Poisson bimodule} over $A$ is a left Poisson module over $A^e$. Denote by $A^e\text{-PMod}$ the category of all Poisson $A$-bimodules. Similarly, there is an equivalence of categories $A^e\text{-PMod}\equiv \U(A)^{e}\text{-Mod}$.

The method of localization is available for Poisson algebras. We list some of the results needed later in this paper.
\begin{lem}\label{Local}
Let $A$ be a Poisson algebra. Suppose $S$ is a multiplicative set of $A$, $M,N$ are two Poisson $A$-modules, and $L$ is a Poisson $A$-bimodule.
\begin{enumerate}
\item The Poisson structure of $A$ extends uniquely to $AS^{-1}$. The set $S$ is an Ore set of $\U(A)$ and $\U(AS^{-1})\cong \U(A)S^{-1}$.
\item The localization $MS^{-1}$ is a Poisson $AS^{-1}$-module.
\item $\Ext_{\U(A)}^*(M,L)$ is equipped with a right Poisson structure coming from $L$. And $\Ext_{\U(A)}^*(M,L)S^{-1}\cong \Ext_{\U(AS^{-1})}^*(MS^{-1},LS^{-1})$ as right Poisson $AS^{-1}$-modules.
\item For any Poisson module morphism $f: M\to N$ and $p\in \Spec(A)$, the localization $f_p:M_p\to N_p$ is a Poisson $A_p$-module morphism . Moreover, $f$ is injective (resp. surjective) if and only if $f_p$ is injective (resp. surjective) for all $p\in \Spec(A)$.
\item For any $A$-module morphism $f: M\to N$, $f$ is a Poisson module morphism if and only if $f_p$ is a Poisson module morphism for all $p\in \Spec(A)$.
\end{enumerate}
\end{lem}
\begin{proof}
All of the items can be checked directly. For (e), consider the $A$-submodule of $N$ spanned by elements $\{a,f(m)\}_N-f(\{a,m\}_M)$ for all $a\in A,m\in M$. Denote it by $K$. Then the condition implies that $K_p=0$ for all $p\in \Spec(A)$. Hence $K=0$ and $f$ is a Poisson module morphism.
\end{proof}

Recall that the module $\Omega_A$ of K\"ahler differentials of $A$ is equipped with the Lie-Rinehart algebra \cite[\S 2]{Rinehart} structure derived from the Poisson structure of $A$. As a consequence, it is observed, by several authors \cite{LWZ, Towers}, that $\U(A)$ is canonically isomorphic to the Lie-Rinehart enveloping algebra $V(A,\Omega_A)$. For a detailed account of this isomorphism, one can refer to \cite[Proposition 5.7]{LWZ}. The advantage of identifying $\U(A)$ with $V(A, \Omega_A)$ is that standard results from the theory of Lie-Rinehart algebras can be applied to $\U(A)$. For example, the algebra $V(A, \Omega_A)$ carries a filtration which naturally passes to $\U(A)$ via the canonical isomorphism such that
\begin{align}\label{FAPE}
\mathscr F_{0}=M_A,\ \mathscr F_{1}=M_A+H_A,\ \mathscr F_{i}=(\mathscr F_{1})^i,\ \text{for}\ i\ge 2.
\end{align}
\begin{prop}\label{PBWP}\cite[Theorem 3.1]{Rinehart}
If $A$ is an affine smooth Poisson algebra, then there is an $A$-algebra isomorphism
\begin{align*}
\gr_{\mathscr F} \U(A)\cong S_A(\Omega_A),
\end{align*}
where $S_A(\Omega_A)$ is the symmetric $A$-algebra on $\Omega_A$. In particular for any $p\in \Spec(A)$, $\gr_{\mathscr F} \U(A_p)\cong A_p[dx_1,dx_2,\dots,dx_\ell]$, where $dx_1,\cdots,dx_\ell$ is a local free basis for $\Omega_{A_p}$ over $A_p$.
\end{prop}
Therefore the algebra $\U(A)$ inherits nice ring-theoretic and homological properties from  $\gr_{\mathscr F} \U(A)$ by the standard results of filtered algebras \cite{HO}.
\begin{cor}\label{PropAPE}
Suppose $A$ is an affine smooth Poisson algebra. Then $\U(A)$ is projective over its subalgebra $A$. Moreover, $\U(A)$ is a noetherian, Auslander-regular domain which is a maximal order in its quotient division ring of fractions.
\end{cor}
\begin{proof}
By Proposition \ref{PBWP}, it is clear that the properties in the statement are satisfied for the localization of the associated graded algebra $\gr_\mathscr F\U(A)$ at any $p\in \Spec(A)$. An easy exercise of localization shows that the same properties hold for $\gr_\mathscr F\U(A)$ as well. Then one can apply the standard results of Zariskian filtrations \cite{HO}.
\end{proof}

\subsection{Poisson homology and Poisson cohomology}
Let $M$ be a right Poisson $A$-module. Then there is a chain complex on the $A$-module $M\otimes_A \Omega^n_A$, where $\Omega^n_A=\wedge^n \Omega_A$ denotes the module of K\"ahler differential $n$-forms. The boundary operator $\partial_n: M\otimes_A\Omega_A^n\to M\otimes_A \Omega^{n-1}_A$ is defined by
\begin{align}\label{PHom}
&\partial_n(m\otimes da_1\wedge \cdots \wedge da_n)=\sum_{1\le i\le n} (-1)^{i+1}\{m,a_i\}_M \otimes da_1\wedge \cdots \wedge \widehat{da_i}\wedge \cdots \wedge da_n\\
&\quad\quad+\sum_{1\le i<j\le n} (-1)^{i+j} m\otimes d\{a_i,a_j\}\wedge da_1\wedge \cdots \wedge \widehat{da_i}\wedge \cdots \wedge \widehat{da_j}\wedge \cdots \wedge da_n.\notag
\end{align}
It is easy to see that $\partial_{n-1}\partial_n=0$. The homology of this complex is denoted by $\HP_*(M)$ and is called the {\it Poisson homology} of the Poisson algebra $A$ with coefficients in the Poisson module $M$ \cite{Mas}.

On the other hand, denote by $\mathfrak{X}^n(M)$ the space of all skew-symmetric $n$-linear maps $\wedge^n A\to M$ that are derivations in each argument. Then there is a cochain complex $(\mathfrak{X}^*(M),\delta^*)$, where $\delta^n: \mathfrak{X}^n(N)\to \mathfrak{X}^{n+1}(N)$ is defined by
\begin{align*}
&\delta^n(f)(a_0\wedge a_1\wedge \cdots \wedge a_n)=\sum_{0\le i\le n} (-1)^{i+1}\{f(a_0\wedge \cdots \wedge \widehat{a_i}\wedge \cdots \wedge a_n), a_i\}_M\\
&\quad\quad +\sum_{0\le i<j\le n} (-1)^{i+j} f(\{a_i,a_j\}\wedge a_0\wedge \cdots \wedge \widehat{a_i}\wedge \cdots \wedge \widehat{a_j}\wedge \cdots \wedge a_n)
\end{align*}
for all $f\in \mathfrak{X}^n(M)$. One sees that $\delta^n$ is well-defined. The cohomology of this complex is denoted by $\HP^*(M)$ and is called the {\it Poisson cohomology} of the Poisson algebra $A$ with coefficients in the Poisson module $M$ \cite{Lic77, Hue}.

When $A$ is affine smooth, then Poisson homology and cohomology can be interpreted as torsion and extension groups via the enveloping algebra $\U(A)$. First of all, $A$ can be viewed as a left $\U(A)$-module as well as a right $\U(A)$-module through the natural Poisson structure on $A$. By \cite[Lemma 4.1]{Rinehart}, the complex $(\U(A)\otimes_A \Omega^*_A,\partial_*)$ with differentials given by \eqref{PHom} ($M=\U(A)$) is a projective resolution of $A$ in the category $\U(A)\text{-Mod}$. The following proposition, to our knowledge, is first explicitly spelled out in \cite{Hue}.

\begin{prop}\label{Poissonext}
Let $A$ be an affine smooth Poisson algebra and $M$ be a right Poisson $A$-module. Then
\begin{equation*}
\HP_*(M)\cong \Tor_*^{\U(A)}(M, A),\quad \HP^*(M)\cong \Ext^*_{\U(A)^{op}}(A, M).
\end{equation*}
\end{prop}

\subsection{Calabi-Yau algebra and rigid dualizing complex}\label{S:Rigid}
Let $B$ be an associative algebra, and $B^e=B\otimes B^{op}$. Let $M$ be a $B$-bimodule, or equivalently, a left $B^e$-module. For every pair of algebra automorphisms $\sigma,\tau$ of $B$, we write $\!^\sigma M^\tau$ for the $B$-bimodule defined by $r\cdot m \cdot s = \sigma(r)m\tau(s)$ for all $r,s\in B$ and $m\in M$. When one or the other of $\sigma,\tau$ is the identity map, we shall simply omit it, writing for example $M^\tau$ for $\!^1M^\tau$. The Van den Bergh condition will constitute a key hypothesis in Poincar\'e duality between Poisson and Hochschild (co)homology.

\begin{deff}\cite[Theorem 1]{VDB2}\label{VDBcondition}
Suppose that $B$ has finite injective dimension $d$ on both sides. Then $B$ satisfies the {\it Van den Bergh condition} if
\begin{align*}
\Ext_{B^e}^i( B,B^e)=
\begin{cases}
0   &   i\neq d\\
U &  i=d,
\end{cases}
\end{align*}
where $U$ is an invertible $B$-bimodule, i.e., there exists another $B$-bimodule $U^{-1}$ such that $U\otimes_BU^{-1}\cong B\cong U^{-1}\otimes_BU$ as $B$-bimodules.
\end{deff}
The definition of Calabi-Yau algebra is due to Ginzburg \cite{Gin}.

\begin{deff}\label{CY}
We say $B$ is {\it skew Calabi-Yau} (skew CY) of dimension $d$ if the following conditions hold:
\begin{enumerate}
\item $B$ is homologically smooth, that is, $B$ has a projective resolution in the category $B^e\text{-Mod}$ that has finite length and such that each term in the
projective resolution is finitely generated, and
\item there is an algebra automorphism $\nu$ of $B$ such that $B$ satisfies the Van den Bergh condition with $U=B^\nu$.
\end{enumerate}
In this case, $\nu$ is said to be the {\it Nakayama automorphism} of $B$ (up to some inner automorphism). Moreover, we say $B$ is {\it Calabi-Yau} (CY) if $\nu$ can be chosen as an inner automorphism.
\end{deff}

Denote by $\text{D}(B^e\text{-Mod})$ ($\text{D}^{\text{b}}(B^e\text{-Mod})$) the (bounded) derived category of all $B$-bimodules. There is a close relationship between $B$ is skew CY and the existence of a rigid dualizing complex in $\text{D}^{\text{b}}(B^e\text{-Mod})$. The next definition is due to Yekutieli \cite{Yek}.
\begin{deff}
Let $B$ be a left and right noetherian algebra. A complex $R\in \text{D}^{\text{b}}(B^e\text{-Mod})$ is called a {\it dualizing complex} over $B$ if it satisfies the following conditions:
\begin{enumerate}
\item $R$ has finite injective dimension over $B$ and over $B^{op}$ respectively.
\item $R$ is homologically finite over $B$ and over $B^{op}$ respectively.
\item The canonical morphisms $B \to \RHom_B(R,R)$ and $B\to \RHom_{B^{op}}(R,R)$ are isomorphisms in $\text{D}(B^e\text{-Mod})$.
\end{enumerate}
\end{deff}
The next definition is due to Van den Bergh \cite{VDB}.
\begin{deff}
Let $B$ be a left and right noetherian algebra. A dualizing complex $R$ is {\it rigid} if
\[
R\cong \RHom_{B^e}\left(B,\!_B R\otimes R_B\right)
\]
in $\text{D}(B^e\text{-Mod})$. The notations $\!_BR$ and $R_B$ mean that we take the $\RHom$ over the left and the right $B$-structures of $R$ respectively.
\end{deff}

The rigid dualizing complex, if it exists, is unique up to isomorphism \cite[Proposition 8.2]{VDB}. The next result is motivated by \cite[Proposition 8.4]{VDB} and \cite[Proposition 4.3]{BZ}.

\begin{prop}\label{CYRigid}
Let $B$ be a left and right noetherian algebra. Then the Van den Bergh condition holds if and only if $B$ has a rigid dualizing complex $V[s]$, where $V$ is invertible and $s\in \mathbb Z$. In this case $U =V^{-1}$ and $s=d$. Moreover if $B^e$ is also noetherian and has finite global dimension. Then $B$ is skew CY of dimension $d$ if and only if $B$ has a rigid dualizing complex $B^{\sigma}[s]$ for some $\sigma\in \Aut(B)$ and $s\in \mathbb Z$. In this case, $d=s$ and $\sigma^{-1}$ is the Nakayama automorphism of $B$.
\end{prop}
\begin{proof}
The first part is exactly \cite[Proposition 4.3]{BZ}. Regarding the second part, we assume $B$ to be skew CY of dimension $d$. It is well known that $B$ has finite global dimension. Then $B$ satisfies the Van den Bergh condition with $U=B^{\nu}$, where $\nu$ is the Nakayama automorphism of $B$. Thus it follows from the first part. Conversely, say $B$ has a rigid dualizing complex $B^{\sigma}[s]$. Still by the first part, $B$ satisfies Definition \ref{CY} (b).  By the assumptions of $B^e$, we can find a projective resolution of $B$ in the category $B^e\text{-Mod}$ such that it is of finite length and each term of the resolution is finitely generated. Hence $B$ is skew CY.
\end{proof}

Suppose $B$ is an affine smooth commutative algebra of dimension $d$. By \cite[Example 3.2.1]{Gin}, we obtain
\begin{align}\label{SerreB}
\Ext_{B^e}^i(B,B^e)=
\begin{cases}
0  &  i\neq d\\
\wedge^d\, \Der_k(B)&  i=d,
\end{cases}
\end{align}
as $B$-modules. Note that in Serre duality, the $B$-bimodule $\Ext_{B^e}^d(B,B^{e})$ is often called the invertible Serre bimodule. Then one deduces that $B$ satisfies the Van den Bergh condition and it  has rigid dualizing complex $\Hom_A(\wedge^d \Der_k(B),A)=\wedge^d\, \Omega_B[d]$.

\begin{prop}\label{ComCY}
Let $B$ be an affine commutative algebra. Then the following are equivalent.
\begin{enumerate}
\item $B$ is skew CY.
\item $B$ is CY.
\item $B$ is smooth and has trivial canonical bundle.
\item $B$ is smooth and has rigid dualizing complex $B[d]$ for some $d\in \mathbb Z$.
\end{enumerate}
\end{prop}
\begin{proof}
(a)$\Leftrightarrow$(b) is based on \cite[Proposition 4.4 (b)]{BZ}.

(b)$\Rightarrow$(c) Clearly $B$ has finite global dimension, and hence it is smooth. By \eqref{SerreB}, $B$ is CY implies that $\wedge^d \Der_k(B)\cong B$ where $d=\dim B$. Then $\wedge^d\Omega_B=\Hom_B(\wedge^d \Der_k(B),B)\cong B$. This means that $B$ has trivial canonical bundle.

(c)$\Rightarrow$(d) It follows from the fact that the rigid dualizing complex of $B$ is given by $\wedge^d\, \Omega_B[d]$, for $d=\dim B$.

(c)$\Rightarrow$(e) We apply Proposition \ref{CYRigid} by using the fact that $B\otimes B$ is noetherian smooth (see \cite[lemma 1]{VDB3}).
\end{proof}

\section{Twisted Poisson module structure}\label{S:Twisted}
In the remaining of the paper, we assume $A$ to be an affine smooth Poisson algebra of dimension $\ell$. The differential forms of maximal degree of $A$ is denoted by $\omega_A=\wedge^\ell \Omega_A$. It is clear that $\wedge^\ell\, \Der_k(A)=\Hom_A(\omega_A,A):=\omega_A^*$, Some of our results hold more generally for arbitrary Poisson algebras, but we will not state them with a specification.

\begin{deff}
A {\it Poisson derivation} of $A$ is a derivation $\delta\in\Der_k(A)$ satisfying
\[\delta\{a, b\}=\{\delta(a),b\}+\{a,\delta(b)\}\]
for any $a,b\in A$. In particular, a Poisson derivation given by $u^{-1}\{u,-\}$ for some $u\in A^\times$ is called a {\it log-Hamiltonian derivation}.
\end{deff}
We denote by $\Der_P(A)$ the set of all Poisson derivations of $A$. For any $u,v\in A^\times$, one sees that $u^{-1}\{u,-\}+v^{-1}\{v,-\}=(uv)^{-1}\{uv,-\}$. Hence the set of all log-Hamiltonian derivations of $A$ forms an additive subgroup of $\Der_P(A)$. We use $\mathscr{HP}^1(A)$ to denote the quotient group of all Poisson derivations modulo log-Hamiltonian derivations. Note that $\mathscr{HP}^1(A)$ differs from the $1$-th Poisson cohomology $\HP^1(A)=\Der_P(A)/\{\text{Hamiltonian derivations}\}$.

In the following, we write $\Id_{\mathscr F}\U(A)$ as the set of all automorphisms $\sigma$ of $\U(A)$ satisfying: (i) $\sigma$ preserves the standard filtration \eqref{FAPE} on $\U(A)$, and (ii) $\sigma=\Id$ when passing to the associated graded algebra $\gr_{\mathscr F}\U(A)$. The next lemma shows that any Poisson derivation of $A$ can be derived from such automorphisms of $\U(A)$, where log-Hamiltonian derivations correspond to those inner automorphisms.

\begin{lem}\label{ModAut}
There is a bijection between
\[
\xymatrix{
\Der_P(A) \ar@<+3pt>[rr]^{\varphi}  && \Id_{\mathscr F}\U(A)\ar@<+3pt>[ll]^{\phi}
}
\]
given by $\varphi(\delta)(M_a)=M_a,\varphi(\delta)(H_a)=H_a+M_{\delta(a)}$ and $M_{\phi(f)(a)}=f(H_a)-H_a$ for any $\delta\in \Der_P(A)$, $f\in \Id_{\mathscr F}(\U(A))$ and $a\in A$. Moreover, the following are bijective:
\begin{enumerate}
\item inner automorphisms of $\U(A)$;
\item inner automorphisms of $\U(A)$ in $\Id_{\mathscr F}\U(A)$;
\item log-Hamiltonian derivations of $A$.
\end{enumerate}
\end{lem}
\begin{proof}
Since $M_A$ and $H_A$ are generators of $\U(A)$, any automorphism $f\in \Id_{\mathscr F}\U(A)$ is given by some linear map $\delta: A\to A$ such that $f(M_a)=M_a, f(H_a)=H_a+M_{\delta(a)}$ for all $a\in A$. Using Definition \ref{UnivP}, it is straightforward to show that $f$ is well-defined if and only if $\delta\in \Der_P(A)$.

By Proposition \ref{PBWP}, $\gr_{\mathscr F} \U(A)$ is a domain. Hence the units of $\U(A)$ belong to $\mathscr F_0 \U(A)=A$, which are exactly the units of $A$. Hence any inner automorphism $f$ of $\U(A)$ is given by
\begin{align*}
f(M_a)&=M_{u}M_aM_{u^{-1}}=M_a\\
f(H_a)&=M_{u}H_aM_{u^{-1}}=(H_aM_u-M_{\{a,u\}})M_{u^{-1}}=H_a+M_{u^{-1}\{u,a\}}
\end{align*}
for some $u\in A^\times$. Hence $f\in \Id_{\mathscr F}\U(A)$ and is given by the log-Hamiltonian derivation such that $f=\varphi(u^{-1}\{u,-\})$. The inverse correspondence  can be proved similarly.
\end{proof}
As a consequence, there is a one-to-one correspondence between the following.
\begin{align}\label{DerInner}
\xymatrix{
\mathscr {HP}^1(A)\ar@<+3pt>[rr]^-{\varphi} && \Id_{\mathscr F}\U(A)/\{\text{inner automorphisms}\}\ar@<+3pt>[ll]^-{\phi}.
}
\end{align}

Let $M$ be a left $\U(A)$-module, and $\sigma$ an automorphism of $\U(A)$. We consider the twisted $\U(A)$-module $\!^\sigma M$. Lemma \ref{ModAut} implies that if $\sigma \in\Id_{\mathscr F}\U(A)$, i.e., it is given by some $\delta\in\Der_P(A)$, then in $\!^\sigma M$ we have
\[
M_a\cdot m=M_am,\quad   H_a\cdot m=H_am+M_{\delta(a)}m
\]
for all $a\in A$ and $m\in M$. Applying the equivalence of categories $A\text{-PMod} \equiv \U(A)\text{-Mod}$, it means that we can twist any Poisson module $(M,\cdot,\{-,-\}_M)$ by some $\delta\in \Der_P(A)$ such that
\[
a\cdot_\delta m=a\cdot m,\  \{a,m\}_\delta=\{a,m\}_M+\delta(a)\cdot m.
\]
In this case, we simply write $\!^\delta M$ as the twisted Poisson module $(M,\cdot_\delta,\{-,-\}_\delta)$. Note that it provides another explanation of \cite[Proposition 2.7]{LWW}. The following lemma shows that twisted Poisson module structure occurs naturally. We will use the forgetful functor from $A\text{-PMod}$ to $A\text{-Mod}$.

\begin{lem}\label{L:APDer}
Let $M$ be a left (resp. right) Poisson $A$-module.  If $M\cong A$ as $A$-modules, then $M\cong \!^\delta A$ (resp. $M\cong A^\delta$) for some $\delta\in \Der_P(A)$. Moreover, the class of $\delta$ in $\mathscr{HP}^1(A)$ is uniquely determined by $M$.
\end{lem}
\begin{proof}
Without loss of generality, we assume $M$ to be a left Poisson $A$-module. We can identify $M$ with $Am$ by choosing some generator $m\in M$. Hence it establishes an isomorphism $\varphi: M\to A$ of left $A$-modules given by $\varphi(am)=a$. For any $a\in A$, there is a unique element $x\in A$ such that $\{a,m\}_M=xm$. Denote $\delta(a)=x$. It is straightforward to check that $\delta\in \Der_P(A)$ and $M\cong \!^\delta A$ as Poisson modules via $\varphi$.

Now let $m'\in M$ be another generator of $M$. Similarly, we have $M\cong \!^{\delta'}A$ where $\delta'\in \Der_P(A)$ is defined by $am'=\delta'(a)m'$ for any $a\in A$. After writing $m'=um$ for some $u\in A^\times$, it is an easy exercise to show that $\delta-\delta'=u^{-1}\{u,-\}$. This implies that $\delta$ is uniquely determined up to some log-Hamiltonian derivation. Hence the class of $\delta$ is unique in $\mathscr{HP}^1(A)$.
\end{proof}

Now we list some results of the behavior of Poisson modules under the tensor and Hom functors in the category $A\text{-Mod}$ via the forgetful functor. These results hold more generally for modules over Lie-Rinehart algebras \cite[pp. 111-112]{Hue99} and for any Lie algebroid \cite[Proposition 4.2.1]{Chem04}. We also state their derived versions.

\begin{prop}\label{P:THPoisson}
\
\begin{itemize}
\item[(a)] Let $M$ be a right Poisson $A$-module and $N$ be a left Poisson $A$-module. Then $M\otimes_AN$ is a right Poisson $A$-module via
\begin{align*}
(m\otimes n)a=(ma)\otimes n=m\otimes (an),\quad \{ m\otimes n, a\}=\{m, a\}_M\otimes n-m\otimes \{a,n\}_N
\end{align*}
for any $m\in M,n\in N$ and $a\in A$. Moreover, there are two left derived functors
\[
M\otimes_A^L(-): \text{D}^{\text{b}}(A\text{-PMod})\to \text{D}^{\text{b}}(A^{op}\text{-PMod}),\, (-)\otimes_A^LN: \text{D}^{\text{b}}(A^{op}\text{-PMod})\to \text{D}^{\text{b}}(A^{op}\text{-PMod}).
\]
\item[(b)] Let $M,N$ be two right Poisson $A$-modules. Then $\Hom_A(M,N)$ is a left Poisson $A$-module via
\begin{align*}
(a\phi)(m)=\phi(ma)=\phi(m)a,\quad \{a,\phi\}(m)=\phi(\{m,a\}_M)-\{\phi(m),a\}_N
\end{align*}
for any $m\in M, a\in A$ and $\phi\in \Hom_A(M,N)$. Moreover, there are two right derived functors
\[
\RHom_A(M,-): \text{D}^{\text{b}}(A^{op}\text{-PMod})\to \text{D}^{\text{b}}(A\text{-PMod}),\, \RHom_A(-,N): \text{D}^{\text{b}}(A^{op}\text{-PMod})\to \text{D}^{\text{b}}(A\text{-PMod}).
\]
\end{itemize}
\end{prop}
\begin{proof}
We only prove (b) and (a) follows in the same fashion. First of all, it is straightforward to see that the Poisson left $A$-module structure is well-defined on $\Hom_A(M,N)$ in the sense of Definition \ref{PMod}. Generally speaking, let $I^\bullet$ be a acyclic complex consisting of injective modules in the category $\U(A)^{op}\text{-Mod}$. For each term $I^i$ in the complex $I^\bullet$, we have
\[\Hom_A(-,I^i_A)=\Hom_A(-,\Hom_{\U(A)}(\, \!_A\U(A),I^i))=\Hom_A(-\otimes_A\U(A),I^i).
\]
Since $\U(A)$ is projective hence flat over $A$ by Corollary \ref{PropAPE}, one sees that $\Hom_A(-,I^i_A)$ is an exact functor. Hence $I^\bullet$ a acyclic complex consisting of injective modules in the category $A\text{-Mod}$ via the forgetful functor. Hence the complex $\Hom_A(M,I^\bullet)$ is acyclic after applying $\Hom_A(M,-)$ to $I^\bullet$. Thus the right derived functor of $\Hom_A(M,-)$ exists by \cite[Theorem 5.1]{Hart}. The argument for $\RHom_A(-,N)$ is same.
\end{proof}

\begin{cor}\label{P:PMTwist}
Let $\delta_1,\delta_2\in \Der_P(A)$ be two Poisson derivations of $A$.
\begin{itemize}
\item[(a)] Let $M$ be a right Poisson $A$-module and $N$ be a left Poisson $A$-module. Then
\[(M^{\delta_1})\otimes_A^L (\, \!^{\delta_2} N)\cong (M\otimes_A^LN)^{(\delta_1-\delta_2)}\]
in $\text{D}^{\text{b}}(A^{op}\text{-PMod})$.
\item[(b)] Let $M,N$ be two right Poisson $A$-modules. Then
\[\RHom_A(M^{\delta_1},N^{\delta_2})\cong\, \!^{(\delta_1-\delta_2)}\, \RHom_A(M,N)\]
in $\text{D}^{\text{b}}(A\text{-PMod})$.
\end{itemize}
\end{cor}
\begin{proof}
We will prove (b) and (a) follows similarly. First of all, we show the isomorphism on the Hom level. Simply write $X=\Hom_A(M^{\delta_1},N^{\delta_2})$, $Y=\!^{(\delta_1-\delta_2)}\Hom_A(M,N)$ and $Z=\Hom_A(M,N)$. It is clear that we have natural isomorphisms of $A$-modules $X\cong Z\cong Y$. For any $a\in A$, $\phi\in X$ and $m\in M^{\delta_1}$, we have
\begin{align*}
\{a,\phi\}_X(m)=&\, \phi(\{m,a\}_{\delta_1})-\{\phi(m),a\}_{\delta_2}=\phi(\{m,a\}_M+m\delta_1(a))-\left(\{\phi(m),a\}_N+\phi(m)\delta_2(a)\right)\\
=&\, \phi(\{m,a\}_M)-\{\phi(m),a\}_N+\phi(m)(\delta_1-\delta_2)(a)=\{a,\phi\}_Z(m)+[(\delta_1-\delta_2)(a)\phi](m)\\
=&\, \{a,\phi\}_Y(m).
\end{align*}
Hence $X\cong Y$ as left Poisson $A$-modules.

More generally, let $N\to I^\bullet$ be an injective resolution of $N$ in the category $\U(A)^{op}\text{-Mod}$. Note that $I$ is injective in the category $\U(A)^{op}\text{-Mod}$ if and only if $I^\delta$ is injective in the category $\U(A)^{op}\text{-Mod}$ for some $\delta\in \Der_P(A)$. Hence $N^{\delta_2}\to (I^\bullet)^{\delta_2}$ is an injective resolution of $N^{\delta_2}$ in the category $\U(A)^{op}\text{-Mod}$. By the argument above, we have the following isomorphisms in  $\text{D}^{\text{b}}(A\text{-PMod})$.
\begin{align*}
\RHom_A(M^{\delta_1},N^{\delta_2})&\, =\RHom_A(M^{\delta_1},(I^\bullet)^{\delta_2})=\Hom_A(M^{\delta_1},(I^\bullet)^{\delta_2})\\
&\, =\!^{(\delta_1-\delta_2)}\Hom_A(M,I^\bullet)=\!^{(\delta_1-\delta_2)}\RHom_A(M,N).
\end{align*}

\end{proof}

\section{Homological unimodularity}\label{S:HomM}
We still suppose $A$ is an affine smooth Poisson algebra of dimension $\ell$. In this section, we study the module of differential forms of maximal degree for $A$, that is $\omega_A=\wedge^\ell\, \Omega_A$. Note that $\omega_A$ is a locally free $A$-module of rank one. It is a well-known fact that $\omega_A$ is equipped with a Poisson $A$-module structure, where any element $H_a\in \U(A)$ acts on $\omega_A$ as a Lie derivation by the adjoint action; see \cite{Bo, Chem94}. We will explain this Poisson structure using homological algebra.

Let $M$ be a Poisson $A$-bimodule, or equivalently, a $\U(A)$-bimodule. The left and right Poisson brackets on $M$ are given by $H_am$ and $mH_a$ for any $m\in M$ and $a\in A$. The following lemma lies in the same fashion of Proposition \ref{P:THPoisson}.
\begin{lem}\label{PAExt}
Let $M$ be a Poisson $A$-bimodule. Then $\Hom_{A^e}(A,M)$ is a right Poisson $A$-module via
\begin{align*}
m\cdot a=ma=am,\quad \{m,a\}=mH_a-H_am
\end{align*}
for any $m\in \Hom_{A^e}(A,M)$ and $a\in A$. Moreover, there exists a right derived functor
\[\RHom_{A^e}(A,-): \text{D}^{\text{b}}(A^e\text{-PMod})\to \text{D}^{\text{b}}(A^{op}\text{-PMod}).\]
\end{lem}
\begin{proof}
It is direct to check that the right Poisson module structure is well-defined on $\Hom_{A^e}(A,M)=\{m\in M\,|\,am=ma\, \forall\, a\in A\}$ with respect to Definition \ref{PMod}. Now consider a acyclic complex $I^\bullet$ consisting of injective objects in the category $\U(A)^e\text{-Mod}$. By Corollary \ref{PropAPE}, $\U(A)^e=\U(A^e)$ is projective hence flat over $A^e$ via the forgetful functor. Hence the complex $\Hom_{A^e}(A,I^\bullet)$ is acyclic. Thus the right derived functor of $\Hom_{A^e}(A,-)$ exists by \cite[Theorem 5.1]{Hart}.
\end{proof}

Note that the tensor algebra $A^e=A\otimes A^{op}$ is a Poisson $A$-bimodule, or equivalently, a left module over $\U(A^e)=\U(A)^e$ by Lemma \ref{TensorU}. Therefore we can take $M=A^e$ in Lemma \ref{PAExt}. It is clear that the right Poisson bracket on $\Hom_{A^e}(A,A^e)$ is given by the adjoint action. Next apply \eqref{SerreB} to get
\[
\Ext_{A^e}^i(A,A^e)=
\begin{cases}
0  &  i\neq \ell\\
\wedge^\ell\, \Der_k(A)=\omega_A^*&  i=\ell
\end{cases}.
\]
Hence $\omega_A^*$ is a right Poisson $A$-module, where the Poisson bracket is induced by the adjoint action. It follows from Proposition \ref{P:THPoisson} (b), one sees that $\omega_A=\Hom_A(\omega_A^*,A)$ is a left Poisson $A$-module.

Now choose any $p\in \Spec(A)$. We know $\Omega_{A_p}$ is a free $A_p$-module of rank $\ell$, and $\omega_{A_p}\cong A_{p}$ as $A_{p}$-modules. Hence $\omega_{A_p}\cong \!^{\delta} A_{p}$ for some $\delta\in \Der_P(A_p)$ by Lemma \ref{L:APDer}.
\begin{deff}\label{LMD}
The {\it local modular derivation} of $A$ at some $p\in \Spec(A)$ is defined to be the Poisson derivation $\delta\in \Der_P(A_p)$ such that $\omega_{A_p}\cong \!^\delta A_p$ as left Poisson $A_p$-modules. Moreover, the {\it local modular class} of $A$ is the class of $\delta$ in $\mathscr{HP}^1(A_p)$.
\end{deff}
As a consequence, the local Poisson structure on $\omega_A$ is uniquely determined by the local modular class of $A$.

\begin{lem}\label{LMDF}
Let $A$ be an affine smooth Poisson algebra, and $p\in \Spec(A)$. Suppose $x_1,x_2,\dots,x_\ell$ is a regular system of parameters of $A_p$. Then the local modular derivation $\delta$ of $A$ at $p$ is uniquely determined, up to some log-Hamiltonian derivation, by its values on $x_1,\dots,x_\ell$. Moreover, we have $\delta(x_i)=\sum_{1\le j\le \ell} a_{ijj}$, where  coefficients $a_{ijk}\in A_p$ are given by $d\{x_i,x_j\}=\sum_{1\le k\le \ell}a_{ijk}dx_k$ in the module of K\"ahler differentials of $A_p$.
\end{lem}
\begin{proof}
Note that $A_p$ is a regular local algebra, whose unique maximal ideal is generated by its regular system of parameters $x_1,x_2,\dots,x_\ell$. Since every element of $A_p$ can be represented as a rational function in terms of $x_1,x_2,\dots,x_\ell$, it is clear that any derivation of $A_p$ is uniquely determined by its values on $x_1,\dots,x_\ell$. This applies to the local modular derivation $\delta\in \Der_P(A_p)$.

Since $dx_1,\dots,dx_\ell$ forms a free basis for the K\"ahler differential $\Omega_{A_p}$ of $A_p$, one sees that $dx_1\wedge dx_2\wedge \cdots \wedge dx_\ell$ is a basis for the differential forms of maximal order $\omega_A$; and hence is a dual basis for $\omega_{A_p}^*$. By applying the standard Koszul resolution of $A_p$, one concludes that the local modular derivation $\delta$ is determined, up to some log-Hamiltonian derivation, by the adjoint action such that
\begin{align*}
\delta(a) dx_1\wedge dx_2\wedge \cdots \wedge dx_\ell=\sum_{1\le i\le \ell} (-1)^{i+1} d(\{a,x_i\}\, dx_1\wedge \cdots \wedge \widehat{dx_i}\wedge \cdots \wedge dx_\ell)
\end{align*}
for any $a\in A_p$. In particular, we have
\begin{align*}
\delta(x_i) dx_1\wedge dx_2\wedge \cdots \wedge dx_\ell&\, =\sum_{1\le j\le \ell}  (-1)^{j+1} d(\{x_i,x_j\}\, dx_1\wedge \cdots \wedge \widehat{dx_j}\wedge \cdots \wedge dx_\ell)\\
&\, =\sum_{1\le j\le \ell} (-1)^{j+1} d(\sum_{k=1}^\ell a_{ijk}\, x_kdx_1\wedge \cdots \wedge \widehat{dx_j}\wedge \cdots \wedge dx_\ell)\\
&\, =\sum_{1\le j\le \ell} (-1)^{j+1} a_{ijj}\,dx_j\wedge dx_1\wedge \cdots \wedge \widehat{dx_j}\wedge \cdots \wedge dx_\ell\\
&\, =(\sum_{1\le j\le \ell}  a_{ijj})\, dx_1\wedge dx_2\wedge \cdots \wedge dx_\ell.
\end{align*}
Hence $\delta(x_i)=\sum_{1\le j\le \ell} a_{ijj}$.
\end{proof}
In particular when $A$ is CY, then $A$ has trivial canonical bundle by Proposition \ref{ComCY}. Hence the Poisson structure on $\omega_A\cong \!^\delta A$ is uniquely determined, up to some log-Hamiltonian derivation, by one single $\delta\in \Der_P(A)$. In \cite[\S 2.2]{LWW}, the Poisson derivation $\delta$ is said to be the {\it modular derivation} of $A$. And the unique class of $\delta$ in $\mathscr{HP}^1(A)$ is called the {\it modular class} of $A$ \cite{Dol}.

\begin{deff}\label{Hunimod}
Let $A$ be an affine smooth Poisson algebra. We say $A$ is {\it homologically unimodular} if $\omega_A\cong A$ as left Poisson $A$-modules, or equivalently, $\omega_A=A$ in the Poisson Picard group $\Pic_P(A)$ of $A$.
\end{deff}

\begin{lem}\label{HomU}
Let $A$ be an affine smooth Poisson algebra of dimension $\ell$. Then the following are equivalent.
\begin{enumerate}
\item $A$ is homologically unimodular.
\item $A$ is CY and is unimodular in the sense of \cite[\S 2.2]{LWW}.
\item $\RHom_{A^e}(A,A^e)=A[-\ell]$ in $\text{D}^{\text{b}}(A^{op}\text{-PMod})$.
\end{enumerate}
\end{lem}
\begin{proof}
(a)$\Rightarrow$(b) We have $\Ext_{A^e}^\ell(A,A^e)=\Hom_A(\omega_A,A)=A$. So $A$ is CY by \Cref{ComCY}. Moreover, one sees that the modular derivation of $A$ ($A$ has trivial canonical bundle) is zero in $\mathscr{HP}^1(A)$. Hence $A$ is unimodular in the sense of \cite[\S 2.2]{LWW}.

(b)$\Rightarrow$(c) and (c)$\Rightarrow$(a) are clear since $\omega_A^* \cong A$ as right Poisson $A$-modules if and only if $\omega_A\cong A$ as left Poisson $A$-modules.
\end{proof}

\section{Poincar\'e duality between Poisson homology and Poisson cohomology}\label{S:PoinP}
In this section, we study Poincar\'e duality between Poisson homology and cohomology following the idea in Van den Bergh duality for Hochschild homology and cohomology \cite{VDB2}. This involves the Serre invertible bimodule $\omega_A^*=\Ext_{A^e}^\ell(A,A^e)$, i.e., the dual module of the differential forms of maximal degree, which plays the same role as the Berezinian module in the general context of Lie-Rinehart algebras \cite{Chem94, Hue99}.

\begin{prop}\label{ExtEq}
Let $A$ be an affine smooth Poisson algebra. Then we have the following isomorphism
\[
\RHom_{\U(A)}(A,\U(A))=\RHom_{A^e}(A,A^e)
\]
in $\text{D}^{\text{b}}(A^{op}\text{-PMod})$.
\end{prop}
\begin{proof}
Throughout our proof, let $p\in \Spec(A)$, and write $B=A_p$ and $R=\gr_\mathscr F \U(B)\cong B[dx_1,\dots,dx_\ell]$ by \Cref{PBWP}. First of all, we show that the left hand side concentrates in degree $\ell$, which reduces our proof to the highest degree $\ell$. We use the standard spectral sequence
$$ \Ext^{*}_{R}(B,R) \Rightarrow \Ext^{*}_{\U(B)}(B, \U(B)).$$
It is easy to see that $\Ext^i_R(B, R)=0$ for $i\neq \ell$. Apply \Cref{Local} to get $\Ext_{\U(B)}^i(B,\U(B))=\Ext_{\U(A)}^i(A,\U(A))_p=0$ for all $p\in \Spec(A)$ if $i\neq \ell$, which implies that $\Ext_{\U(A)}^i(A,\U(A))=0$ for all $i\neq \ell$.

Next we will compute $\Ext_{\U(B)}^\ell(B,\U(B))$. Let $x_1,\dots,x_\ell$ be a regular system of parameters for $B$. We denote by $a_{ijk}\in B$ the coefficients determined by $d\{x_i,x_j\}=\sum_{1\le k\le \ell}a_{ijk}dx_k$. Then the local modular derivation $\delta\in \Der_P(B)$ is given by $\delta(x_i)=\sum_{1\le j\le \ell} a_{ijj}$ by Lemma \ref{LMDF}. In the following, write
\[e=dx_1\wedge dx_2\wedge \cdots \wedge dx_\ell,\quad e_i=dx_1\wedge \cdots \wedge \widehat{dx_i}\wedge \cdots \wedge dx_\ell\quad \text{for}\, 1\le i\le \ell\]
as free bases for $\Omega_B^\ell$ and $\Omega_B^{\ell-1}$. We use the complex $\U(B)\otimes_B \Omega_B^\bullet\to B$ with differentials described in \eqref{PHom} ($M=\U(B)$) as a projective resolution of $B$ in the category $\U(B)\text{-Mod}$. Therefore, we obtain the following commutative diagram of exact rows
\[
\xymatrix{
\Hom_B(\Omega_B^{\ell-1},\U(B))\ar[r]^-{\partial}\ar@{=}[d] &  \Hom_B(\Omega_B^\ell,\U(B))\ar[r]\ar@{=}[d]  &   \Ext_{\U(B)}^\ell(B,\U(B))\ar[r]\ar[r]\ar@{=}[d]  &  0\\
\bigoplus_{i=1}^\ell e_i\, \U(B)\ar[r]^-{\partial} & e\, \U(B) \ar[r]  &   \Ext_{\U(B)}^\ell(B,\U(B))\ar[r] &  0,
}
\]
where the differential $\partial$ can be explicitly given as follows
\begin{align*}
\partial(e_i)/e&\,=\delta_{e_i}(\partial(e))\\
&\,=\delta_{e_i}(\sum_{1\le j\le \ell}(-1)^{j+1}e_jH_{x_j}+\sum_{1\le j<k\le \ell} (-1)^{j+k} d\{x_j,x_k\}\wedge dx_1\wedge \cdots \wedge \widehat{dx_j}\wedge \cdots \wedge \widehat{dx_k}\wedge \cdots \wedge dx_\ell)\\
&\,=(-1)^{i+1}H_{x_i}+\delta_{e_i}(\sum_{\substack{1\le j<k\le \ell\\ 1\le m\le \ell}} (-1)^{j+k} a_{jkm}dx_m\wedge dx_1\wedge \cdots \wedge \widehat{dx_j}\wedge \cdots \wedge \widehat{dx_k}\wedge \cdots \wedge dx_\ell)\\
&\,=(-1)^{i+1}H_{x_i}+\delta_{e_i}(\sum_{1\le j<k\le \ell} (-1)^{j+k} a_{jkj}dx_j\wedge dx_1\wedge \cdots \wedge \widehat{dx_j}\wedge \cdots \wedge \widehat{dx_k}\wedge \cdots \wedge dx_\ell\\
&\quad\quad+ \sum_{1\le j<k\le \ell} (-1)^{j+k} a_{jkk}dx_k\wedge dx_1\wedge \cdots \wedge \widehat{dx_j}\wedge \cdots \wedge \widehat{dx_k}\wedge \cdots \wedge dx_\ell)\\
&\,=(-1)^{i+1}H_{x_i}+\delta_{e_i}(\sum_{1\le j<k\le \ell} (-1)^{k+1} a_{jkj}e_k+ \sum_{1\le j<k\le \ell} (-1)^{j} a_{jkk}e_j)\\
&\,=(-1)^{i+1}H_{x_i}+\sum_{1\le j<i\le \ell} (-1)^{i+1} a_{jij}+ \sum_{1\le i<k\le \ell} (-1)^{i} a_{ikk}\\
&\,=(-1)^{i+1}(H_{x_i}-\sum_{1\le j\le\ell} a_{ijj})\\
&\,=(-1)^{i+1}(H_{x_i}-\delta(x_i)).
\end{align*}
In the above calculation, we use the fact that $a_{ijk}=-a_{jik}$. Note that $\U(B)$ is generated by $B$ and $H_{x_1},\dots, H_{x_\ell}$ by Proposition \ref{PBWP} and the local modular derivation $\delta$ is uniquely determined by $\delta(x_1),\dots,\delta(x_\ell)$ by Lemma \ref{LMDF}. As a conclusion, we have
\[
\Ext_{B^e}^\ell(B,B^e) \cong B^\delta \cong \U(B)/(H_{x_i}-\delta(x_i))_{1\le i\le \ell}\cong \Ext_{\U(B)}^\ell(B,\U(B)).
\]
as right Poisson $B$-modules.

Finally, we define the global $A$-module map $\varphi: \Ext_{A^e}^i(A,A^e)\to  \Ext_{\U(A)}^i(A,\U(A))$ as $\varphi=0$ when $i\neq \ell$ and when $i=\ell$, $\varphi$ is the composition of the following maps.
\[
\xymatrix{
\Ext_{A^e}^\ell(A,A^e)=\Hom_A(\omega_A,A)\ar@{^{(}->}[r] & \Hom_A(\omega_A,\U(A))\ar@{->>}[r] &  \Ext_{\U(A)}^\ell(A,\U(A)).
}
\]
By the arguments above, any localization of $\varphi$ is a Poisson module isomorphism over the localized Poisson algebra. Hence the statement follows from Lemma \ref{Local}, and the derived version is obvious.
\end{proof}

In the theory of Lie-Rinehart algebras, there is a correspondence (analogous to BernsteinÕs correspondence for $\mathcal D$-modules) between left and right Lie-Rinehart modules using the Berezinian module of the dual sheaf of the Lie algebras \cite{Chem94}. We reformulate it in the context of Poisson algebras by using any Poisson $A$-modules which are line bundles over $A$.

\begin{lem}\label{EquivMod}
Let $S$ be a right Poisson $A$-module. If $S$ is a locally free $A$-module of rank one, then the following two functors $F: A\text{-PMod}\to A^{op}\text{-PMod}$ and $G:A^{op}\text{-PMod}\to A\text{-PMod}$ defined by $F(-)=S\otimes_A(-)$ and $G(-)=\Hom_A(S,-)$ with Poisson $A$-modules structures given by Proposition \ref{P:THPoisson} is an equivalence of categories. As a consequence, there is an equivalence of bounded derived categories
\begin{align*}
\xymatrix{
\text{D}^{\text{b}}(A\text{-Mod})\ar@<+3pt>[rr]^-{F} &&\text{D}^{\text{b}}(A^{op}\text{-Mod})\ar@<+3pt>[ll]^-{G}.
}
\end{align*}
\end{lem}
\begin{proof}
Since $S$, considered as a right $A$-module, is locally free of rank one, it is straightforward to check that $GF=\Id$ on $A\text{-PMod}$ and $FG=\Id$ on $A^{op}\text{-PMod}$ by applying the method of localization in \Cref{Local}. The derived version follows immediately.
\end{proof}
When $S=A$, we will simply identify $(-)=A\otimes_A(-)$ and $(-)=\Hom_A(A,-)$ between left and right Poisson modules. For instance, the original right Poisson module $\omega_A^*$ can be considered as a left Poisson module, where $\Hom_A(A, \omega_A^*)=\omega_A^*$ as vector spaces and $am=ma$, $\{a,m\}=-\{m,a\}$ for any $a\in A$ and $m\in \omega_A^*$. In the remaining of the paper, we will freely switch between left and right Poisson modules when we apply Proposition \ref{P:THPoisson}. In particular, it is an easy exercise to show that the corresponding left module for a twisted right Poisson module $M^\delta$ is given by $\!^{-\delta}M$ and vice visa.

\begin{thm}\label{TwistedP}
Let $A$ be an affine smooth Poisson algebra of dimension $\ell$. Then for any right Poisson $A$-module $M$, we have
\[
\HP^{i}(M)=\HP_{\ell-i}(M\otimes_A \omega_A^*).
\]
In particular if $A$ is homologically unimodular, then we have the Poincar\'e duality between Poisson homology and cohomology
\[
\HP^{i}(M)=\HP_{\ell-i}(M).
\]
\end{thm}
\begin{proof}
In view of Proposition \ref{Poissonext}, it is convenient to make use of the derived category.
\begin{align*}
\HP^i(M)&\, =\Homology^i(\RHom_{\U(A)^{op}}(A,M))=\Homology^i(\RHom_{\U(A)}(A,M))\\
&\, =\Homology^i(\RHom_{\U(A)}(A,\U(A))\otimes_{\U(A)}^L M)=\Homology^i(\RHom_{A^e}(A,A^e)\otimes_{\U(A)}^L M)\\
&\,=\Homology^i(\omega_A^*[-\ell]\otimes_{\U(A)}^L M)=\Homology^{i-\ell}(\omega_A^*\otimes_{\U(A)}^L M)=\Homology^{i-\ell}(M\otimes_{\U(A)}^L \omega_A^*)\\
&\, =\Homology^{i-\ell}((M\otimes_A^L \omega_A^*)\otimes_{\U(A)}^L A)=\HP_{\ell-i}(M\otimes_A\omega_A^*).
\end{align*}
We use the fact that $M\otimes_{\U(A)}\omega_A^*\cong (M\otimes_A \omega_A^*)\otimes_{\U(A)}A$ and $(-)\otimes_A \omega_A^*: \U(A)^{op}\text{-Mod}\to \U(A)^{op}\text{-Mod}$ is an equivalence of categories by a left module version of Lemma \ref{EquivMod}.

In particular if $A$ is homologically unimodular, then $\omega_A\cong A$ as left Poisson modules. Hence $\omega_A^*=\Hom_A(\omega_A,A)\cong A$ as left Poisson modules, which implies the Poincar\'e duality.
\end{proof}

The twisted Poincar\'e duality between Poisson homology and cohomology was studied  by Launois-Richard \cite{LR07} for polynomial Poisson algebras with quadratic Poisson structures. Following their ideas, Zhu \cite{Zhu} obtained a twisted Poincar\'e duality for affine Poisson algebras, and later it was proved for any polynomial Poisson algebras with values in an arbitrary Poisson module by Luo-Wang-Wu \cite{LWW}. The following result provides a generalization of the twisted Poincar\'e duality for all CY Poisson algebras.

\begin{cor}\label{C:TwistD}
Let $A$ be a CY Poisson algebra of dimension $\ell$, and $M$ a right Poisson $A$-module. Then we have the twisted Poincar\'e duality between Poisson homology and cohomology
\[
\HP^{i}(M)=\HP_{\ell-i}(M^\delta),
\]
where $\delta$ is the modular derivation of $A$.
\end{cor}
\begin{proof}
Now $A$ is CY. By Proposition \ref{ComCY} and Lemma \ref{L:APDer}, we know $\omega_A\cong \!^\delta A$ as left Poisson modules, where $\delta$ is the modular derivation of $A$. By a right version of Proposition \ref{P:THPoisson} (b), one sees that $\omega_A^*=\Hom_A(\!^\delta A,A)\cong A^\delta$ as right Poisson modules. In view of Lemma \ref{EquivMod} ($S=A$), we have $\omega_A^*\cong \!^{-\delta}A$ as left Poisson modules. Now it is an easy exercise to check that $M\otimes_A \!^{-\delta}A\cong M^\delta$ as right Poisson modules following Proposition \ref{P:THPoisson} (a). Hence $\HP^i(M)=\HP_{\ell-i}(M\otimes_A \!^{-\delta}A)=\HP_{\ell-i}(M^\delta)$ by Theorem \ref{TwistedP}.
\end{proof}

\section{Calabi-Yau condition}\label{S:CY}
In this section, we use the rigid dualizing complex of the enveloping algebra $\U(A)$ to explore its Calabi-Yau property. We show that there is a strong connection between the unimodularity of $A$ and the Calabi-Yau property of $\U(A)$.

\begin{lem}\label{GradedCY}
For any affine smooth Poisson algebra $A$ of dimension $\ell$, the associated graded algebra $\gr_{\mathscr F}\U(A)$ has a rigid dualizing complex
\[
\gr_{\mathscr F}\U(A)\otimes_A (\omega_A\otimes_A\omega_A)[2\ell].
\]
In particular if $A$ is CY, then $\gr_{\mathscr F}\U(A)$ is CY.
\end{lem}
\begin{proof}
It suffices to show for $B:= S_A(\Omega_A)$ by Proposition \ref{PBWP}. We define the following map
\[\varphi: \Omega_B\to B\otimes_A\left(\Omega_A\oplus \Omega_A\right)\]
first on the generators of $\Omega_B$ such that $\varphi(da)=1\otimes (da+0)$ and $\varphi(d\alpha)=1\otimes (0+\alpha)$ for any $a\in A$ and $\alpha\in\Omega_A$. Then one checks that $\varphi$ can be extended to a well-defined $B$-module map on $\Omega_B$. Now consider $\varphi$ as an $A$-module map. The localization of $B$ at any $p\in \Spec(A)$ implies that $B_p=S_{A_p}(\Omega_{A_p})\cong A_p[dx_1,\dots,dx_\ell]$ by Proposition \ref{PBWP}. Hence $\Omega_{B_p}$ is a free $B_p$-module of rank $2\ell$ and the localization $\varphi_p$ yields $\varphi_p: B_p^{\oplus 2\ell}\to B_p^{\oplus 2\ell}$, which is easy to be checked as an isomorphism. Therefore $\varphi$ is an isomorphism of $A$-modules; and hence it is an isomorphism of $B$-modules. Since $B$ is affine smooth, by the comment above Proposition \ref{ComCY}, the rigid dualizing complex of $B$ is given by $2\ell$ shifting of
\begin{align*}
\wedge^{2\ell}\, \Omega_B&\, =\wedge^{2\ell} B\otimes_A(\Omega_A\oplus \Omega_A)=B\otimes _A \wedge^{2\ell}(\Omega_A\oplus \Omega_A)\\
&\,=B\otimes _A (\wedge^{\ell}\Omega_A \otimes _A \wedge^\ell \Omega_A)=B\otimes_ A (\omega_A\otimes_A\omega_A).
\end{align*}
Moreover, when $A$ is CY, we know $\omega_A=A$, then it follows from Proposition \ref{ComCY} since the rigid dualizing complex of $B$ is $B[2\ell]$.
\end{proof}

Generally, suppose $S$ is a Poisson $A$-module that is a locally free $A$-module of rank one. In the following, we denote by $S^*=\Hom_A(S,A)$ the dual module of $S$, which is also a locally free $A$-module of rank one; and its has again a Poisson module structure due to Proposition \ref{P:THPoisson}. We will not specify on which side is the Poisson module structure of $S$ or $S^*$, but will leave to the context making use of the comment below Lemma \ref{EquivMod}. The next result verifies \Cref{D:PicardG}.

\begin{lem}\label{RankonePMod}
Let $S$ or $S_i$ be Poisson modules over $A$, which are locally free $A$-modules of rank one. We have the following isomorphisms of Poisson $A$-modules.
\begin{enumerate}
\item $S\otimes A\cong S\cong A\otimes_AS$;
\item $S\otimes_A S^*\cong A\cong S^*\otimes_A S$;
\item $(S_1\otimes_A S_2)^*\cong S_2^*\otimes_A S_1^*$;
\item $\Hom_A(S_1,S_2)\cong S_1^*\otimes_A S_2$;
\item $\Hom_A(S_1,S_2)^*\cong \Hom(S_2,S_1)$.
\end{enumerate}
Moreover if $U$ is an invertible Poisson bimodules over $A$, then $(S\otimes_A U)^{-1}=U^{-1}\otimes_A S^*$.
\end{lem}
\begin{proof}
For any $p\in \Spec(A)$, the localization $S_p\cong A_p$  is equipped with the Poisson module structure given by some local Poisson derivation $\delta\in \Der_P(A_p)$ by Lemma \ref{L:APDer}. Hence we can prove all the claims by localization. For instance in (b), there is natural map $\varphi: S\otimes_A S^*\to A$ given by evaluation. Let $S_p\cong A^\delta_p$ for some local Poisson derivation $\delta\in\Der_P(A_p)$. By Corollary \ref{P:PMTwist}, we have
\begin{align*}
\varphi_p: S_p \otimes_A \Hom_{A_p}(S_p,A_p)&\, =A_p^{\delta} \otimes_A \Hom_{A_p}(A_p^\delta,A_p)=A_p^{\delta} \otimes_A \!^\delta\, \Hom_{A_p}(A_p,A_p)\\
&\, =A_p\otimes_A \Hom_{A_p}(A_p,A_p)\to A_p
\end{align*}
is clearly an isomorphism of Poisson modules. Hence $\varphi$ is an isomorphism of Poisson modules by Lemma \ref{Local}. We can prove $S^*\otimes_A S\cong A$ analogously.
\end{proof}

Before we state our results, let us set up the convection for the structure of Poisson bimodules concerning the tensor product $S\otimes_A \U(A)$ (resp. $\U(A)\otimes_A S$) for any Poisson module $S$ that is a line bundle over $A$.  We require that the right Poisson structure of $S\otimes \U(A)$ (resp. $\U(A)\otimes S$) is derived from the right (resp. left) multiplication of $\U(A)$ and the left (resp. right) $\U(A)$-module structure is determined by the tensor product with possible switching side of Poisson module structures regarding $S$ and Proposition \ref{P:THPoisson} by applying the comment below Lemma \ref{EquivMod}.

Now according to \Cref{EquivMod}, one sees that $(S\otimes_A \U(A))[d]$ and $(\U(A)\otimes_A S)[d]$ for any $d\in \mathbb Z$ are all dualizing complexes in $\text{D}^{\text{b}}(\U(A)^e\text{-Mod})$. As suggested by Lemma \ref{GradedCY}, the line bundle $S=\mathcal L_A$ such that
\begin{align}\label{CYdualmod}
\mathcal L_A:=\omega_A\otimes_A\omega_A=\Hom_A(\omega_A^*,\omega_A)
\end{align}
plays a significant role in the rigid dualizing complex of $\U(A)$. Clearly, $\mathcal L_A$ is equipped with a left Poisson module structure with respect to Proposition \ref{P:THPoisson} (b) when we treat both $\omega_A^*$ and $\omega_A$ as right Poisson modules. By Lemma \ref{RankonePMod}, the dual module of $\mathcal L_A$ is given by
\begin{align}\label{CYdualdualmod}
\mathcal L_A^*=\omega_A^*\otimes_A\omega_A^*=\Hom_A(\omega_A,\omega_A^*).
\end{align}
In particular when $A$ is CY, then $\mathcal L_A=\Hom_A(A^{\delta},A^{-\delta})=\!^{2\delta }A$ as left Poisson modules, or $\mathcal L_A=A^{-2\delta}$ as right Poisson modules, where $\delta$ is the modular derivation of $A$.

\begin{prop}\label{VDBP}
For any affine smooth Poisson algebra $A$ of dimension $\ell$, we have
\[
\RHom_{\U(A)^{op}}(A,\RHom_{A^e}(A,\U(A)^e))=\U(A)\otimes_A\mathcal L_A^*[-2\ell]
\]
in $\text{D}^{\text{b}}(\U(A)^e\text{-Mod})$.
\end{prop}
\begin{proof}
In view of \Cref{PAExt}, it is easy to check that the left hand side above is well-defined. Since $\U(A)^{e}=\U(A^e)$ is flat over $A^e$ by Corollary \ref{PropAPE}, we have
\begin{align*}
\RHom_{A^e}(A,\U(A)^e)&\, =\RHom_{A^e}(A,A^e\otimes_{A^e} \U(A)^e)\\
&\, =\RHom_{A^e}(A,A^e)\otimes_{A^e}^L \U(A)^e\\
&\, =\omega_A^*\,[-\ell]\otimes_{A^e}^L \U(A)^e\\
&\, =(\U(A)\otimes_A \omega_A^*)\otimes_A \U(A)\, [-\ell].
\end{align*}
Note that for the last equality above, the right $\U(A)$-module structure is determined by Proposition \ref{P:THPoisson} (b) when $\omega_A^*$ is considered as a left $\U(A)$-module. Next we get
\begin{align*}
\RHom_{\U(A)^{op}}(A,\RHom_{A^e}(A,\U(A)^e))&\, =\RHom_{\U(A)^{op}}(A,(\U(A)\otimes_A \omega_A^*)\otimes_A \U(A)\,[-\ell])\\
&\, =(\U(A)\otimes_A \omega_A^*)\otimes_A^L\RHom_{\U(A)^{op}}(A, \U(A))\,[-\ell]\\
&\, =(\U(A)\otimes_A \omega_A^*)\otimes_A (\omega_A^*\,[-\ell])\,[-\ell]\\
&\, =(\U(A)\otimes_A \omega_A^*)\otimes_A \omega_A^*\,[-2\ell]\\
&\, =\U(A)\otimes_A (\omega_A^*\otimes_A \omega_A^*)\,[-2\ell]\\
&\, =\U(A)\otimes_A \mathcal L_A^*\,[-2\ell].
\end{align*}
%
\end{proof}

The following result has been proved in a general setting for Lie-Rinehart Lie (super)algebras \cite{Chem94, Chem99}. We reformulate its Poisson version making it compatible with our notations.
\begin{thm}\label{RigidDCP}
For an affine smooth Poisson algebra $A$ of dimension $\ell$, the rigid dualizing complex of $\U(A)$ is
\[
\U(A)\otimes_A\mathcal L_A\, [2\ell].
\]
In particular if $A$ is CY, the rigid dualizing complex of $\U(A)$ is $\U(A)^{-2\delta}\,[2\ell]$, where $\delta$ is the modular derivation of $A$.
\end{thm}
\begin{proof}
We compare the notations with \cite{Chem99} concerning the Lie-Rinehart pair $(A,\Omega_A)$. The affine variety is $X=\Spec(A)$. The sheaf of Lie algebras is given by $\mathcal L_X=\widetilde{\Omega_A}$, which is a locally free $\mathcal O_X$-module of rank $d_{\mathcal L_X}=\ell$. The anchor map $\omega: \mathcal L_X\to \Theta_X$ is given by $da\mapsto \{a,-\}$ for any $a\in A$. The sheaf of differential operators $\mathcal D(\mathcal L_X)$ is the $\mathcal O_X$-algebra $\widetilde{\U(A)}$ and we have
\[
\wedge^{d_{\mathcal L_X}} \mathcal L_X^*=\left(\wedge^\ell \Hom_A(\Omega_A,A)\right)^{\widetilde{}}=\widetilde{\omega_A^*}.
\]
Note that both $\wedge^{d_{\mathcal L_X}} \mathcal L_X^*$ and $\omega_X=\omega_A$ are $\mathcal D(\mathcal L_X)$-modules, where any $D\in \mathcal L_X$ acts as a Lie derivation by adjoint action. Hence we apply \cite[Theorem 4.4.1]{Chem99} to conclude that the rigid dualizing complex of $\U(A)$ is given by
\[\U(A)\otimes_A \mathcal \Hom_A(\omega_A^*,\omega_A)\,[2\ell]=\U(A)\otimes_A \mathcal L_A\,[2\ell].\]

Now suppose $A$ is CY. By the comment above \Cref{VDBCY}, $\mathcal L_A\cong \!^{2\delta}A$ as left Poisson $A$-modules, where $\delta$ is the modular derivation of $A$. Hence we have
\[
\U(A)\otimes_A \mathcal L_A\,[2\ell]=\U(A)\otimes_A  \!^{2\delta}A\, [2\ell]=\U(A)^{-2\delta}\,[2\ell].
\]
\end{proof}

\begin{cor}\label{VDBCY}
For any affine smooth Poisson algebra $A$ of dimension $\ell$, we have
\[
\Ext_{\U(A)^e}^i(\U(A),\U(A)^e)=
\begin{cases}
0   &   i\neq 2\ell\\
\mathcal L_A^*\otimes_A\U(A)&  i=2\ell
\end{cases}
\]
as $\U(A)$-bimodules.
\end{cor}
\begin{proof}
Since $\U(A)$ is noetherian, we can apply \Cref{CYRigid} to obtain the Van den Bergh invertible bimodule $U=(\U(A)\otimes_A \mathcal L_A)^{-1}=\mathcal L^*_A\otimes_A \U(A)$ by \Cref{RankonePMod}.
\end{proof}

As a consequence, the Calabi-Yau property of the enveloping algebra $\U(A)$ follows from the unimodularity of $A$ when $A$ is CY.
\begin{cor}\label{SCYP}
Let $A$ be a CY Poisson algebra of dimension $\ell$. Then $\U(A)$ is skew CY of dimension $2\ell$. Moreover, the Nakayama automorphism of $\U(A)$ is given by $2\delta$, where $\delta$ is the modular derivation of $A$.
\end{cor}
\begin{proof}
We apply Proposition \ref{CYRigid} in terms of Theorem \ref{RigidDCP} using the fact that $\U(A)^e=\U(A\otimes A)$ is noetherian smooth.
\end{proof}

\begin{question}
Consider the following diagram.
\[
\xymatrix{
\U(A)^e\text{-Mod}\ar[rrrr]^-{\Hom_{\U(A)^e}(\U(A),-)\quad\quad}\ar[drr]_-{\Hom_{A^e}(A,-)} & &&& \text{Ab}\\
&& \U(A)^{op}\text{-Mod}\ar[rru]_-{\quad\quad \Hom_{\U(A)^{op}}(A,-)}&&
}.
\]
It is clear to check that
\[
\Hom_{\U(A)^e}(\U(A),M)=\Hom_{\U(A)^{op}}(A,\Hom_{A^e}(A,M)),
\]
for any $\U(A)$-bimodule $M$. But the functor $\Hom_{A^e}(A,-)$ does not preserve the injective objects in general. Is it possible to derive \Cref{VDBCY} from \Cref{VDBP} directly without using the general results in \cite{Chem99}?
\end{question}

Now we can state the twisted Poincar\'e duality for Hochschild homology and cohomology regarding Poison bimodules.
\begin{cor}\label{VDBDualityH}
Let $A$ be an affine smooth Poisson algebra of dimension $\ell$. For any Poisson $A$-bimodule $M$, we have
\begin{align*}
\HH^i(M)=\HH_{2\ell-i}(\mathcal L_A^*\otimes M),
\end{align*}
where $\HH_i(M)$ (resp. $\HH^i(M)$) denotes the $i$-th Hochschild homology (resp. cohomology) of $\U(A)$ with coefficients in $M$. Moreover if $A$ is CY, then
\begin{align*}
\HH^i(M)=\HH_{2\ell-i}(\!^{-2\delta}M),
\end{align*}
where $\delta$ is the modular derivation of $A$.
\end{cor}
\begin{proof}
The result is an application of \cite[Theorem 1]{VDB2} where we use the Van den Bergh condition in terms of Lemma \ref{VDBCY}.
\end{proof}

\begin{thm}\label{EqHUCY}
Let $A$ be a CY Poisson algebra of dimension $\ell$. Then the following are equivalent.
\begin{enumerate}
\item $A$ is unimodular.
\item The module class of $A$ is zero.
\item $\omega_A\cong A$ as left Poisson modules over $A$.
\item $\omega_A^*\cong A$ as right Poisson modules over $A$.
\item $\Ext_{\U(A)}^\ell (A,\U(A))\cong A$ as right Poisson modules over $A$.
\item $\Ext_{\U(A)^{op}}^\ell (A,\U(A))\cong A$ as left Poisson modules over $A$.
\item $\Ext_{A^e}(A,A^e)\cong A$ as right Poisson modules over $A$.
\item $\U(A)$ is CY of dimension $2\ell$ provided that for any $u\in A^\times$, $\sqrt{u}$ exits.
\end{enumerate}
\end{thm}
\begin{proof}
(a)$\Leftrightarrow$(b) comes from definition. (c)$\Leftrightarrow$(d) and (e)$\Leftrightarrow$(f) follows from duality of left and right Poisson modules regarding Lemma \ref{EquivMod}. (d)$\Leftrightarrow$(e) is derived from Proposition \ref{ExtEq}. Lemma \ref{HomU} implies that (a)$\Leftrightarrow$(e) since $A$ is CY. By Definition \ref{LMD} and the comment below Lemma \ref{LMDF}, we know $\Ext_{A^e}(A,A^e)=\omega_A^*$ as right Poisson modules, hence (d)$\Leftrightarrow$(g). Finally, it suffices to prove (a)$\Leftrightarrow$(h). By Corollary \ref{SCYP}, we know $\U(A)$ is skew CY with Nakayama automorphism given by $2\delta$ in the sense of Lemma \ref{ModAut}, where $\delta$ is the modular derivation of $A$. In this case by \Cref{CY}, $\U(A)$ is CY if and only if $2\delta$ is given by some inner automorphism of $\U(A)$ if and only if $2\delta=u^{-1}\{u,-\}$ for some $u\in A^\times$ if and only if $\delta=v^{-1}\{v,-\}$ for $v=\sqrt{u}$ if and only if $A$ is unimodular.
\end{proof}

\begin{remark}\label{CYUni}
For a polynomial Poisson algebra $A=\k[x_1,\dots,x_n]$, its units are given by $A^\times=\k^\times$. Thus we know a polynomial Poisson algebra is unimodular if and only if its enveloping algebra is CY.  Generally speaking, by the proof of \Cref{EqHUCY}, one sees that the unimodularity of the Poisson structure always implies the CY property of its enveloping algebra, but the inverse direction may not hold.
\end{remark}

\begin{Ack}
The second author is grateful for the hospitality of the first author at Zhejiang Normal University summer 2016 during the time the paper is written . The authors want to thank James Zhang, Xiaolan Yu and Guisong Zhou for helpful suggestions and correspondences. This work was supported by the National Natural Science Foundation of China (Grant Nos. 11571316, 11001245) and the Natural Science Foundation of Zhejiang Province (Grant No. LY16A010003).
\end{Ack}


\medskip
\begin{thebibliography}{99}
\bibitem[BP11]{BP11} R.~Berger and A.~Pichereau, Calabi-Yau algebras viewed as deformations of Poisson algebras, {\it Algebr.~Represent.~Theory}, 17 (2014), 735--773.
\bibitem[Bry88]{Bry88} J.~L.~Brylinski, A differential complex for Poisson manifolds, \emph{J.~Differential~Geom.}, 28 (1988), 93--114.
\bibitem[BrZu99]{BZ99} J.-L.~Brylinski and G.~Zuckerman, The outer derivation of a complex Poisson manifold,\emph{J.~Reine~Angew.~Math.}, 506 (1999) 181--189.
\bibitem[Bo87]{Bo} A.~Borel, Algebraic $D$-modules, Academic press (1987).
\bibitem[BZ08]{BZ} K.~A.~Brown and J.~J.~Zhang, Dualising complexes and twisted Hochschild (co)homology for noetherian Hopf algebras, \emph{J.~Algebra}, 320 (2008), 1814--1850.
\bibitem[Ch94]{Chem94} S.~Chemia, Poincar\'e duality for $k-A$-Lie superalgebras, \emph{Bull.~Soc.~Math.~France}, 122 (1994), 371--397.
\bibitem[Ch99]{Chem99} S.~Chemia, A duality property for complex Lie algebroids, \emph{Math.~Z.}, 232 (1999), 367--388.
\bibitem[Ch04]{Chem04} S.~Chemia, Rigid dualizing complex for quantum enveloping algebras and algebras of generalized differential operators, \emph{J.~Algebra}, 276 (2004), 80--102.
\bibitem[Dol09]{Dol} V.~A.~Dolgushev, The Van den Bergh duality and the modular symmetry of a Poisson variety, \emph{Selecta~Math.}, 14 (2009), 199--228.
\bibitem[EG10]{EG} P.~Etingof and V.~Ginzburg, Noncommutative del Pezzo surfaces and Calabi-Yau algebras, \emph{J.~Eur.~Math.~Soc.}, 12 (2010), 1371--1416.
\bibitem[Gin]{Gin} V.~Ginzburg, Calabi-Yau algebras, preprint, arXiv:math/0612139.
\bibitem[GK93]{GK} V.~Ginzburg and S.~Kumar, Cohomology of quantum groups at roots of unity, \emph{Duke~Math.~J.} 69 (1993), 179--198.
\bibitem[Har66]{Hart} R.~Hartshorne, \{it Residues and Duality\}, in: Lecture Notes in Math., vol. 20, Springer-Verlag, Berlin, 1966.
\bibitem[Hue90]{Hue} J.~Huebschmann, Poisson cohomology and quantization, \emph{J.~Reine~Angew.~Math.}, 408 (1990) 57--113.
\bibitem[Hue99]{Hue99} J.~Huebschmann, Duality for Lie-Rinehart algebras and the modular class, \emph{J.~Reine~Angew.~Math.} 510 (1999), 103--159.
\bibitem[HO96]{HO} L.~Huishi and F.~van~Oystaeyen, Zariskian filtrations, Kluwer Academic Publishers, K-monographs in Mathematics, vol. 2 (1996).
\bibitem[LR07]{LR07} S.~Launois and L.~Richard, Twisted Poincar\'e duality for some quadratic Poisson algebras, \emph{Lett.~Math.~Phys.}, 79 (2007), 161--174.
\bibitem[Lic77]{Lic77} A.~Lichnerowicz, Les varieties de Poisson et leurs algebres de Lie associees (French), \emph{J.~Differential~Geometry}, 12 (1977), 253--300.
\bibitem[LWW]{LWW} J.~Luo, S.-Q.~Wang, Q.-S.~Wu, Twisted Poincar\'e duality between Poisson homology and Poisson cohomology, \emph{J.~Algebra}, 442 (2015), 484--505.
\bibitem[LWZ15a]{LWZ} J.-F.~L\"u, X.~Wang, and G.~Zhuang, Universal enveloping algebras of Poisson Hopf algebras, \emph{J.~Algebra}, 426 (2015), 92--136.
\bibitem[LWZ15b]{LWZ2} J.-F.~L\"u, X.~Wang, and G.~Zhuang, Universal enveloping algebras of Poisson Ore-extensions, \emph{Proc.~Amer.~Math.~Soc.}, 143 (2015), 4633--4645.
\bibitem[LWZ16]{LWZ3} J.-F.~L\"u, X.~Wang, and G.~Zhuang, DG Poisson algebra and its universal enveloping algebra, \emph{Sci.~China~Math.}, 59 (2016), 849--860.
\bibitem[Mas]{Mas} T.~Maszczyk, Maximal commutative subalgebras, Poisson geometry and Hochschild homology, arXiv: math.KT/0603386.
\bibitem[Mar04]{Mar04} N.~Marconnet, Homologies of cubic Artin-Schelter regular algebras, \emph{J.~Algebra}, 278 (2004), 638--665.
\bibitem[Oh99]{Oh} S.-Q.~Oh, Poisson enveloping algebras, \emph{Comm.~Algebra}, 27 (1999), 2181--2186.
\bibitem[Ri63]{Rinehart} G.~S.~Rinehart, Differential forms on general commutative algebras, \emph{Trans.~Amer.~Math.~Soc.}, 108 (1963), 195--222.
\bibitem[Pen83]{Pen} I.~B.~Penkov, $D$-modules on supermanifolds, \emph{Invent.~Math.}, 71 (1983), 501--512.
\bibitem[Tow]{Towers} M.~Towers, Poisson and Hochschild cohomology and the semiclassical limit, arXiv:1304.6003, 2013.
\bibitem[Um12]{Um} U. Umirbaev, Universal enveloping algebras and universal derivations of Poisson algebras, \emph{J.~Algebra}, 354 (2012), 77--94.
\bibitem[VdB94]{VdB94} M.~Van den Bergh, Noncommutative homology of some three-dimensional quantum spaces, \emph{K-Theory} 8 (1994), 213--230.
\bibitem[VdB97]{VDB} M.~Van den Bergh, Existence theorem for dualizing complexes over non-commutative graded and filtered rings, \emph{J.~ Algebra}, 195 (1997), 662--679.
\bibitem[VdB98]{VDB2} M.~Van den Bergh, A relation between Hochschild homology and cohomology for Gorenstein rings, \emph{Proc.~Amer.~Math.~Soc.}, 126 (1998), 1345--1348.
\bibitem[VdB02]{VDB3} M.~Van den Bergh, Erratum to: A relation between Hochschild homology and cohomology for Gorenstein rings, \emph{Proc.~Amer.~Math.~Soc.}, 130 (2002), 2809--2810.
\bibitem[Wei97]{Wein} A.~Weinstein, The modular automorphism group of a Poisson manifold, \emph{J.~Geom.~Phys.}, 23 (1997), 379--394.
\bibitem[Xu99]{Xu99} P.~Xu, Gerstenhaber algebras and BV-algebras in Poisson geometry, \emph{Commun.~Math.~Phys.}, 200 (1999), 545--560.
\bibitem[Ye92]{Yek} A.~Yekutieli, Dualizing complexes over non commutative graded algebras, \emph{J.~Algebra} 153 (1992), 41--84.
\bibitem[Zhu15]{Zhu} C.~Zhu, Twisted Poincar\'e duality for Poisson homology and cohomology of affine Poisson algebras, \emph{Proc.~Amer.~Math.~Soc.}, 143 (2015), 1957--1967.
\end{thebibliography}
\end{document}